\documentclass[12pt]{article}
\makeatletter
\def\@seccntformat#1{\@ifundefined{#1@cntformat}%
   {\csname the#1\endcsname\quad}  
   {\csname #1@cntformat\endcsname}
}
\let\oldappendix\appendix 
\renewcommand\appendix{%
    \oldappendix
    \newcommand{\section@cntformat}{\appendixname~\thesection\quad}
}
\makeatother
\usepackage{latexsym}
\usepackage{amssymb}
\usepackage{amsmath}
\usepackage{amsthm}
\usepackage{rotating}
\usepackage{multirow}
\usepackage{mathrsfs}
\usepackage{wasysym}
\usepackage{stmaryrd}
\usepackage{slashed}
\usepackage{enumerate}
\usepackage{rawfonts}
\input{prepictex}
\input{pictex}
\input{postpictex}
\usepackage[OT2,OT1]{fontenc}
\def\cyr{%
\renewcommand\rmdefault{wncyr}%
\renewcommand\sfdefault{wncyss}%
\renewcommand\encodingdefault{OT2}%
\normalfont
\selectfont}
\DeclareMathAlphabet{\zap}{OT1}{pzc}{m}{it}
\DeclareTextFontCommand{\textcyr}{\cyr}
\newcommand{\rad}{\text{\cyr   ya}}
\newcommand{\rid}{\text{\cyr   i}}
\def\CC{\mathbb C}

\newtheorem{main}{Theorem}

\DeclareMathOperator{\Diff}{\zap{Diff}}
\DeclareMathOperator{\Hom}{Hom}

\DeclareMathOperator{\tr}{tr}

\newtheorem{thm}{Theorem}

\newtheorem{lem}{Lemma}
\newtheorem{prop}{Proposition}

\newtheorem{*corem}{Corollary}

\def\ZZ{{\mathbb Z}}
\def\RR{{\mathbb R}}
\def\CP{{\mathbb C \mathbb P}}
\def\RP{{\mathbb R \mathbb P}}

\begin{document}

\title{Einstein Metrics, Conformal Curvature, and Anti-Holomorphic Involutions}

\author{Claude LeBrun\thanks{Supported 
in part by  NSF grant DMS-1906267.}\\ 
Stony
 Brook
 University} 
  
\date{}

\maketitle

\begin{abstract} 
Building on  previous results  \cite{lebdet,pengwu}, we complete the classification of 
compact oriented Einstein $4$-manifolds with $\det (W^+) > 0$.  There are, up to diffeomorphism,  exactly 
15   manifolds that carry such  metrics, and, on  each of these manifolds,    such  metrics sweep out 
exactly one connected component of the corresponding Einstein moduli space. 
 \end{abstract}

%
%

\section{Introduction}

A Riemannian metric $h$ is said \cite{bes}  to be {\em Einstein} if, for some real constant $\lambda$, 
 it satisfies  the {\em Einstein equation}
\begin{equation*}
\label{einstein}
r= \lambda h,
\end{equation*}
where $r$ is the Ricci tensor of $h$.
Given a smooth compact $n$-manifold $M$,  henceforth  always assumed to be connected and without boundary, 
one  would like   to completely understand the  Einstein moduli
space
$$\mathscr{E}(M) = \{ \mbox{Einstein metrics on  } M\}/(\Diff (M) \times \RR^+),$$
 where the diffeomorphism group $\Diff (M)$  acts on  metrics  via pull-backs, and where the  positive reals $\RR^+$ 
act  by  rescaling.
This moduli problem is well understood \cite{thurston,thur3} in dimensions $n\leq 3$, 
because in these low dimensions the Einstein equation is actually equivalent  to just requiring  the sectional curvature to be  constant. 
By contrast,
when  $n\geq 5$, the  abundance of currently-available examples 
of ``exotic'' Einstein metrics on familiar manifolds \cite{bohm,bgk,wazi}  seems to indicate
 that the problem could very well turn out 
  to be  fundamentally     intractable in high dimensions.  
On the other hand, 
 there are certain   specific $4$-manifolds, such as  real and complex-hyperbolic $4$-manifolds,   the $4$-torus, and $K3$, 
where  the  Einstein  moduli space  $\mathscr{E}(M)$ is  explicitly known, and    in fact turns out to be 
 {connected} \cite{bergb,bcg,hit,lmo}. 
 This  provides clear  motivation for the  intensive  study of  Einstein moduli spaces in dimension four.   
 
The idiosyncratic features of $4$-dimensional Riemannian geometry  are generally  attributable to   the 
 failure of  the  Lie group    $\mathbf{SO}(4)$ to be simple; instead,  its Lie algebra  decomposes as a direct   sum of proper subalgebras:
$$\mathfrak{so}(4) = \mathfrak{so}(3) \oplus  \mathfrak{so}(3).$$
Because $\mathfrak{so}(4)$   and  $\wedge^2\RR^4$ can both be realized as  the space of skew $4 \times 4$   matrices,
this  leads to a natural decomposition 
$$
\Lambda^2 = \Lambda^+ \oplus \Lambda^-  
$$
of 
the bundle of $2$-forms on an oriented Riemannian $4$-manifold $(M,h)$. Since  the sub-bundles $\Lambda^\pm$
coincide with the $(\pm 1)$-eigenspaces  of 
the Hodge star 
operator $\star: \Lambda^2 \to \Lambda^2$,   sections of $\Lambda^+$ are called 
self-dual $2$-forms, while sections of 
$\Lambda^-$ are  called 
anti-self-dual $2$-forms. But because  
 the  Riemann curvature tensor can be naturally   identified with 
 a self-adjoint linear map 
$${\mathcal R} : \Lambda^2 \to \Lambda^2,$$ 
the curvature of $(M^4, h)$ can   consequently be  decomposed into  four   pieces 
$$
{\mathcal R}=
\left(
\mbox{
\begin{tabular}{c|c}
&\\
$W^++\frac{s}{12}I$&$\mathring{r}$\\ &\\
\cline{1-2}&\\
$\mathring{r}$ & $W^-+\frac{s}{12}I$\\&\\
\end{tabular}
} \right) ,
$$
corresponding to different irreducible representations of  $\mathbf{SO}(4)$. 
Here   $s$ is the {scalar curvature} and 
$\mathring{r}$ is  the trace-free   
      Ricci curvature, while 
$W^\pm$ are  by definition  the trace-free pieces of the appropriate blocks.
The corresponding pieces  ${W^{\pm a}}_{bcd}$ of the Riemann curvature tensor 
are in fact both conformally invariant, and 
 are  respectively called the
{\em self-dual} and {\em anti-self-dual Weyl curvature} tensors. 
The sum $W=W^++W^-$ is called the {\em Weyl tensor} or {\em conformal curvature} tensor, and vanishes if and only if the
metric $h$ is locally conformally flat. It should  be emphasized    that the distinction between the 
self-dual and anti-self-dual  parts of the 
Weyl tensor depends on a choice of  orientation; 
 reversing the orientation of $M$ interchanges $\Lambda^+$ and $\Lambda^-$, and so  interchanges
$W^+$ and $W^-$, too.

The present paper is a natural outgrowth of previous work on   the Einstein moduli spaces $\mathscr{E}(M)$ 
 for the   smooth compact oriented  $4$-manifolds $M$  that  arise  as {\sf del Pezzo surfaces}. 
 Recall that a {del Pezzo surface}  is defined to be  a 
 compact complex surface $(M^4,J)$  with ample    anti-canonical 
line bundle.   Up to diffeomorphism, there are
exactly ten such manifolds, namely  $S^2 \times S^2$ and the nine connected sums $\CP_2\# m\overline{\CP}_2$,  $m = 0, 1, \ldots, 8$. 
These are exactly \cite{chenlebweb}  the    smooth oriented compact  
$4$-manifolds that  admit both an Einstein metric with $\lambda > 0$ and an orientation-compatible symplectic structure.  
However, 
the currently-known Einstein metric on any of these spaces   are all conformally K\"ahler. 
Indeed, on most del Pezzos,   the currently-known Einstein metrics   \cite{sunspot,tian} are actually  K\"ahler-Einstein,
although  there are   two exceptional cases where  they are instead non-trivial conformal rescalings of special 
extremal K\"ahler metrics  \cite{chenlebweb,lebhem10}. Inspired in part by earlier work by Derdzi\'nski \cite{bes,derd}, and 
building upon his own results  in \cite{lebhem,lebuniq},
the author was eventually able  to characterize  \cite{lebcake} the known Einstein metrics on 
del Pezzo manifolds  by the property  that 
$W^+(\omega , \omega ) > 0$ everywhere, where $\omega$ is a non-trivial   (global) self-dual harmonic $2$-form. 
An interesting   corollary is that the known Einstein metrics on each  del Pezzo $4$-manifold $M$ 
exactly sweep out  one    connected component of  the corresponding 
Einstein moduli space $\mathscr{E}(M)$.

However, the  role of a global harmonic $2$-form $\omega$ in the above  criterion   makes it disquietingly   non-local. 
Fortunately,  Peng Wu \cite{pengwu}  has  recently discovered that  these known Einstein metrics can instead be characterized by
demanding    that 
$\det (W^+)$ be positive at every point, where the self-dual Weyl curvature is considered as an endomorphism 
$$W^+:\Lambda^+ \to \Lambda^+$$
of the rank-$3$ bundle of self-dual $2$-forms. 
 The present author  then found \cite{lebdet} an entirely different 
 proof  of this characterization  that  actually strengthens  the result, while 
 at the same time highlighting the previously-neglected  point that this criterion   only 
 forces our compact oriented Einstein manifold to be a del Pezzo if we explicitly  require it to  be
  {\em simply connected}. 
In this paper, we will  tackle  this last issue head-on,  
by describing  the moduli space
$$\mathscr{E}_{\det} (M) = \{ \mbox{Einstein metrics on $M$ with } \det (W^+) > 0 \}/(\Diff (M) \times \RR^+)$$
for each compact oriented $4$-manifold $M$ where this moduli space is non-empty. Our first main 
result  is the following: 

 \begin{main} 
\label{alpha} 
There are exactly 15 diffeotypes of compact oriented $4$-manifolds $M$ that carry 
Einstein metrics $h$ with $\det (W^+) > 0$  everywhere. 
For each such manifold, the moduli space $\mathscr{E}_{\det}(M)$ of these special Einstein metrics is 
connected, and   exactly sweeps out a single  connected component of the Einstein moduli space 
$\mathscr{E} (M)$.
\end{main}

In order to state our second, more detailed  main result, we will  first need to  consider  two different $\ZZ_2$-actions on 
$S^2\times S^2$. Let $\mathfrak{a}: S^2 \to S^2$ denote the antipodal map, and let $\mathfrak{r}: S^2 \to S^2$
denote reflection across the equator, so that
\begin{equation}
\label{reflect} 
\mathfrak{a} = \left[\begin{array}{ccc}-1 &  &  \\ & -1 & \\ &  & -1\end{array}\right]
\qquad \mbox{and} \qquad
\mathfrak{r} =
 \left[\begin{array}{ccc}1 &  &  \\ & 1 & \\ &  & -1\end{array}\right]
\end{equation}
as elements of $\mathbf{O}(3)$. 
Then $\mathfrak{a} \times \mathfrak{r}$  and 
$\mathfrak{a} \times \mathfrak{a}$ are 
 both 
  free, orientation-preserving
involutions of $S^2 \times S^2$, and  the smooth compact $4$-manifolds 
\begin{eqnarray}
\mathscr{P}&:=& (S^2 \times S^2)/\langle \mathfrak{a} \times \mathfrak{r} \rangle  \nonumber \\
\mathscr{Q} &:=& (S^2 \times S^2)/\langle \mathfrak{a} \times \mathfrak{a} \rangle \label{mind}
\end{eqnarray}
are therefore both orientable. Note, however, these two manifolds are not  even homotopy equivalent   \cite[p. 101]{hakr}, 
because $\mathscr{P}$ is spin, whereas $\mathscr{Q}$ is not.

\begin{main} 
\label{beta} 
Let $M$ be a smooth compact oriented $4$-manifold that  is {\sf not simply connected}.
Then, in 
 the notation defined by \eqref{mind},  $M$ admits an an Einstein metric $h$ with     $\det (W^+) > 0$   if and only if
$M$ is diffeomorphic  to  $\mathscr{P}$ or    to 
$\mathscr{Q} \# k \overline{\CP}_2$ for some 
$k =0, 1, 2, 3$. Moreover, whenever such an Einstein metric $h$ exists, 
the universal cover  $(\widetilde{M}, \widetilde{h})$  of $(M,h)$ is necessarily isometric to a del Pezzo surface, equipped with 
a K\"ahler-Einstein metric, in such a manner that the non-trivial deck transformation becomes a free anti-holomorphic involution. 
\end{main}

The proofs of these results are given in \S \ref{mainline} below, as the culmination  of  a series of 
 detailed  case-by-case studies   carried  out in earlier parts of \S \ref{double}. 
Then,  in \S \ref{last},  we  conclude the article by  generalizing these results in various ways, while 
also pointing  pointing out some associated  open problems.  
%
%

\section{Del Pezzos and Double Covers} 
\label{double} 

We begin by carefully refining the statement of  \cite[Proposition 2.3]{lebdet} in order to  
emphasize a key 
 technical  fact that lay buried in the proof. 

\begin{prop} 
\label{rule}
Let $(M,h)$ be a compact  oriented  Einstein $4$-manifold which satisfies 
$\det (W^+) > 0$ at every point. Then  either
\begin{enumerate}[{\rm (i)}]
\item $\pi_1(M)=0$, and $M$ admits an orientation-compatible  complex structure $J$ such that 
$(M,J)$  is a del Pezzo surface, and such that  the conformally rescaled  metric $g = |W^+|_h^{2/3}h$ is a  $J$-compatible  K\"ahler metric; or else, \label{mahi} 
\item $\pi_1(M)=\ZZ_2$, and $M$ is doubly covered by a del Pezzo surface   $(\widetilde{M},J)$  on which 
 the pull-back of  $g = |W^+|_h^{2/3}h$ is a  $J$-compatible  K\"ahler metric  $\widetilde{g}$, and where   the non-trivial deck transformation 
 $\sigma: \widetilde{M}\to \widetilde{M}$  is an anti-holomorphic involution of  $(\widetilde{M},J)$. \label{ahi} 
 \end{enumerate} 
\end{prop} 
\begin{proof} The conformal rescaling of $h$ 
used in \cite{lebdet} was actually constructed as  $\alpha_h^{2/3}h$, where
$\alpha_h$ is the top eigenvalue of $W^+_h : \Lambda^+ \to \Lambda^+$. However,  once this rescaled metric has been shown to be
K\"ahler, it then follows that  $-\alpha_h/2$ is a repeated eigenvalue of  $W^+_h$, so that   one necessarily also has 
$|W^+|_h^2 = \frac{3}{2} \alpha_h^2$. Thus, the K\"ahler metric constructed in  \cite{lebdet}  simply coincides, 
up to a constant factor of $\sqrt[3]{3/2}$, with the metric  $g = |W^+|_h^{2/3}h$ considered above. 

The proof of  \cite[Proposition 2.3]{lebdet} actually  focuses on the real line-bundle $L\subset \Lambda^+$ given by the top eigenspace
of $W^+$; this is well-defined, because the  identity  $\tr (W^+) =0$ and the hypothesis $\det(W^+) > 0$  together 
imply that the top eigenvalue  of $W^+$ has multiplicity one everywhere. If $L$ is trivial,  one can then choose a global section
$\omega$ of $L$ such that $|\omega|_g\equiv \sqrt{2}$, and a Weitzenb\"ock argument (made possible by  the fact that any Einstein metric 
satisfies $\delta W^+=0$) is then used to show that $\omega$ is parallel. If, on the other hand, $L$ is non-trivial, 
$\widetilde{M}= \{ \omega\in L~|~ |\omega|_g =  \sqrt{2}\}$ defines a double cover of $M$ that comes
equipped with a tautological self-dual $2$-form $\omega$ that,    by the same Weitzenb\"ock argument as before, 
can then 
 be shown to be the K\"ahler form 
of the pulled-back metric $\widetilde{g}$ with respect to a suitable complex structure $J$. In the latter case, 
the non-trivial deck transformation $\sigma : \widetilde{M}\to \widetilde{M}$ preserves $\widetilde{g}$, and sends 
 $\omega$ to $-\omega$,  and so, because  $\omega = \widetilde{g} ( J \cdot , \cdot )$, must send 
$J$  to $-J$. Thus, in case \eqref{ahi}, 
$\sigma$ is  an anti-holomorphic involution of $(\widetilde{M}, J)$. 

Finally, the complex surface $(M,J)$ or $(\widetilde{M}, J)$ is automatically a del Pezzo. Indeed, since 
any K\"ahler surface  satisfies $\det (W_+) = s^3/864$, 
where $s$ is its scalar curvature, 
the assumption that $\det (W^+) > 0$ implies the scalar curvature 
of $g$ or $\widetilde{g}$ must be  positive  everywhere. 
Since the Einstein metric $h$   therefore has positive Einstein constant, and can now  be rewritten as $24 s^{-2} g$,
the transformation law for the Ricci curvature  under conformal changes implies
 \cite{lebhem} that the $(1,1)$-form $\rho + 2i \partial \bar{\partial} \log s$ is a positive
representative of $2\pi c_1$, where $\rho$ is the Ricci form of  our K\"ahler surface. 
The Kodaira embedding theorem thus implies  that the anti-canonical line-bundle $K^{-1}$  is 
 ample, and  $(M,J)$ or $(\widetilde{M}, J)$
is  therefore a del Pezzo surface, as claimed. 
\end{proof}

Because case \eqref{mahi} was  thoroughly analyzed in  previous papers \cite{lebdet,lebcake,sunspot},
we will only need to carefully discuss case   \eqref{ahi} in this article. Fortunately, this part  of the problem 
can largely be reduced to well-explored questions in real algebraic geometric. Indeed, 
since  $(\widetilde{M}, J)$  
can be embedded in a projective space $\mathbb{P} ([\Gamma (\mathcal{O}(K^{-\ell})]^*)$ on which 
$\sigma$ acts by complex conjugation,  $\widetilde{M}$ can be viewed as a complex projective algebraic variety 
defined over $\mathbb{R}$; and because the action of  $\sigma$  on $\widetilde{M}$ has no   fixed points, 
this variety automatically has empty real locus. The substantial classical and modern literature available concerning
real forms of del Pezzo surfaces \cite{quartics,kolreal,russo,ctcdp} has therefore   paved the road ahead of   us, 
and will make  it comparatively easy to completely solve the problem. 

Since traditional approaches to the subject emphasize  the 
 {\em degree} $c_1^2> 0$ 
of a del Pezzo surface, it will be important for us to relate  the degree of $\widetilde{M}$ to  the topology of
$M= \widetilde{M}/\langle \sigma \rangle$. For this purpose, it is useful to remember that any almost-complex
$4$-manifold satisfies $c_1^2= 2\chi + 3\tau$, 
where $\chi$ is the Euler characteristic and $\tau = b_+ - b_-$ is the signature. 
For the del Pezzo surface $\widetilde{M}$, however, the Todd genus $\mathbf{Td}= h^{0,0} - h^{0,1} + h^{0,2}= (\chi + \tau)/4$
must equal $1$, since $h^{0,1}=h^{0,2}=0$ by the Kodaira vanishing theorem. It therefore follows
that 
$$c_1^2 (\widetilde{M}) = 8 + \tau(\widetilde{M}) = 8 + 2\tau (M),$$
where in the last step we have recalled that  the 
signature $\tau$ is multiplicative under finite covers. 
On the other hand, $b_+(M)=0$, since the K\"ahler form $\omega$ spans the self-dual harmonic forms
on $(\widetilde{M},\widetilde{g})$, but is  $\sigma$-anti-invariant. 
Hence $\tau (M) = -b_-(M) = - b_2(M)$, and 
 $c_1^2 (\widetilde{M})= 2[4- b_2(M)]$.   As  a consequence, the only possibilities  are  $b_2(M)=0,1,2$ or $3$.  
We will now proceed by discussing each of these cases separately.

\subsection{The  $b_2(M)=0$ Case}

When $b_2(M)=0$, the double cover $\widetilde{M}$ must have  signature zero. Since this covering space 
is therefore a del Pezzo surface of degree $8$, classification \cite{delpezzo,dolgachev} tells us  that $\widetilde{M}$ is
 diffeomorphic to either  $S^2\times S^2$ or $\CP_2\# \overline{\CP}_2$. 
Now, it is a classical fact \cite{russo,ctcdp} that any anti-holomorphic involution of the  one-point blow-up of $\CP_2$ 
must  have a fixed point. But, as we will now observe,   this is actually preordained by  a more
general topological result. Although  elementary, the proof is worth recounting here in some detail, 
as doing so will  eventually save us needless extra work 
  in \S \ref{last}. 

\begin{lem} 
\label{halte}
No  smooth {\sf orientable}    $4$-manifold $M$ with $\pi_1\neq 0$ 
has a covering space   homeomorphic to $\CP_2 \# \overline{\CP}_2$. 
\end{lem} 
\begin{proof} Let us  proceed by contradiction, and assume there exists  a covering map
$\varpi : N\to M$, where  $M$ is a smooth   oriented non-simply-connected \linebreak  $4$-manifold,
and where $N$ is homeomorphic (but perhaps not diffeomorphic) to $\CP_2 \# \overline{\CP}_2$. 
Notice  that  $M= \varpi (N)$ is automatically compact, and  that the simply connected manifold
$N$ is automatically its universal cover. We now give  
$N$ the orientation lifted from $M$, so that  the  degree $\geq 2$ of  $\varpi$ then equals  $|\pi_1(M)|$.
Since this in particular means that  $\pi_1(M)$ is  finite, 
$$H^1(M,\RR ) = \Hom (\pi_1 (M) , \RR ) =0,$$
and Poincar\'e duality for the oriented $4$-manifold $M$ therefore implies 
$$b_3 (M) =  b_1 (M) =0 \quad \mbox{and} \quad  b_4 (M) = b_0 (M) = 1, $$
where $b_j$ denotes the $j^{\rm th}$ Betti number with  $\RR$ coefficients. 
The Euler characteristic of $M$ is therefore given by 
$$\chi (M ) = \sum_{j=0}^4 (-1)^j b_j (M)  = 2+ b_2(M)\geq 2.$$
However, because the Euler characteristic $\chi$ is multiplicative under finite coverings,
we also  have
$$\chi (M ) = \chi (\CP_2 \# \overline{\CP}_2)/|\pi_1(M)|= 4 /|\pi_1(M)|\leq 2.$$
It therefore  follows  that $\chi (M)=2$, and that $b_2(M) =0$. 
In particular, $H^2 (M,\ZZ)$  has trivial  free part, and so  consists entirely 
of torsion elements. 

On the other hand, any smooth, orientable $4$-manifold is spin$^c$. 
Thus, there exists \cite{hiho,whitney} an integral cohomology class  $a\in H^2(M,\ZZ)$ 
satisfying 
$$\varrho (a) = w_2 (M) := w_2 (TM),$$
where 
$$\varrho: H^2 (M, \ZZ ) \to  H^2 (M, \ZZ_2)$$
denotes the natural homomorphism  induced by mod-$2$ reduction $\ZZ\to \ZZ_2$. 
However, since $\varpi$ is a smooth submersion,  $\varpi_* : TN  \cong  \varpi^*TM$. Thus,  the naturality
of Stiefel-Whitney classes with respect to pull-backs and  the 
 commutativity of
the diagram 
\begin{eqnarray*}
H^2(N, \ZZ ) &\stackrel{\varrho}{\to} & H^2 (N, \ZZ_2 )\\
   \uparrow_{\varpi^*} &&\quad \uparrow_{\varpi^*} \\
H^2(M, \ZZ ) &\stackrel{\varrho}{\to} & H^2 (M, \ZZ_2 )\
\end{eqnarray*}
together guarantee that 
\begin{eqnarray*}
\varrho \, ( \varpi^*(a) ) &=& \varpi^*( \varrho \,  (a)) =  \varpi^*(w_2 (TM) )\\
&=&  w_2 (\varpi^*TM ) = 
w_2 (TN)
\\&=&   w_2 (N) \in H^2(N, \ZZ_2).
\end{eqnarray*}
On the other hand, since $a\in H^2 (M, \ZZ)$  is a torsion element,  it follows that $\varpi^*a\in H^2 (N, \ZZ)$ is a torsion element, too. 
But  
$$H^2 (N, \ZZ)\cong H^2 ( \CP_2 \# \overline{\CP}_2 , \ZZ) = \ZZ \oplus \ZZ $$  is a free Abelian group,
so this implies  that 
 $\varpi^*a =0$. Hence  $$w_2 (N)= \varrho \, ( \varpi^*(a) )  =0.$$  But this is  absurd, because
$N\approx  \CP_2 \# \overline{\CP}_2$ has odd intersection form, and so is  not spin. 
It follows that the  oriented  $4$-manifold $M$    cannot exist, as claimed. 
\end{proof}

In our context, this simple fact has a striking consequence:

\begin{thm} Let $(M,h)$ be a compact  oriented  {\sf non-simply-connected} Einstein $4$-manifold that  satisfies 
$\det (W^+) > 0$ at every point.  Then $M$ is doubly covered by a del Pezzo surface   $(\widetilde{M},J)$
on which the pull-back $\widetilde{h}$ of $h$ is a $J$-compatible {\sf K\"ahler-Einstein} metric. 
\label{katydid}
\end{thm}
\begin{proof}
Case \eqref{ahi} of Proposition \ref{rule} tells us  that the Einstein manifold  $(\widetilde{M},\widetilde{h})$ is conformally
K\"ahler. However, by \cite[Theorem A]{lebuniq}, $\CP_2\# \overline{\CP}_2$ and $\CP_2\# 2\overline{\CP}_2$ are the only 
 two compact $4$-manifolds that carry 
Einstein metrics that are   conformally K\"ahler,  but not K\"ahler-Einstein.
But  neither of these is  the double
cover of an  oriented $4$-manifold;  the second is prohibited because
 its signature is odd, while 
 the first is ruled out by Lemma \ref{halte}. \end{proof}

As an immediate consequence, 
any compact oriented Einstein $4$-manifold $(M, h)$ with 
$\det (W^+) > 0$ and $b_2=0$ must be   doubly covered by $\CP_1\times \CP_1$, equipped with a K\"ahler-Einstein metric. 
However,  a theorem of 
 Matsushima  \cite[Th\'eor\`eme 1]{mats} implies that  any K\"ahler-Einstein metric on $\CP_1\times \CP_1$
must be invariant under a maximal compact subgroup
 $\cong \mathbf{SO}(3)\times \mathbf{SO}(3)$ of the identity component $\mathbf{PSL}(2,\CC ) \times \mathbf{PSL}(2,\CC )$
 of the complex automorphism group. 
Thus,  the universal cover $(\widetilde{M}, \widetilde{h})$ of $(M,h)$ must be homothetic  to 
the homogeneous Einstein manifold $(S^2 , g_0) \times (S^2 , g_0)$, where $g_0$
is the ``round'' unit-sphere metric on $S^2 = \CP_1$. This   allows us  to  deduce  the following:

\begin{prop} Modulo constant rescalings, any compact oriented Einstein $4$-manifold $(M, h)$ with 
$\det (W^+) > 0$ and $b_2=0$ is isometric to exactly one of the Riemannian  quotients described by 
\eqref{mind}. Since the  two $4$-manifolds $\mathscr{P}$ and $\mathscr{Q}$ are not
diffeomorphic, it thus follows that the  moduli spaces $\mathscr{E}_{\det} (\mathscr{P})$ and $\mathscr{E}_{\det} (\mathscr{Q})$
each   consist of a single point. \label{twofer}
\end{prop}
\begin{proof}
With respect to the product metric $g_0\oplus g_0$, the sectional curvature $K(\Pi)$ of a $2$-plane $\Pi\subset T(S^2 \times S^2)$ 
belongs to $[0,1]$, and satisfies $K(\Pi)=1$ iff $\Pi$ is tangent to an $S^2$ factor. Thus, any isometry of $(S^2 \times S^2, g_0\oplus g_0)$
must send each $2$-sphere $S^2 \times \{ pt\}$ or $\{ pt\} \times S^2$ to a $2$-sphere of one of these two types. On the other hand, 
because the orientation-preserving 
 isometric involution $\sigma : S^2 \times S^2 \to S^2 \times S^2$ must  not have   fixed points, the Lefschetz fixed-point theorem 
 tells us that  its Leftschetz
number must vanish. That is, 
$$0 = L(\sigma ) = \sum_j (-1)^j \tr \left(\sigma_*|_{H_j (S^2 \times S^2)}\right) = 2 + \tr \left(\sigma_*|_{H_2 (S^2 \times S^2)}\right),$$
where $\sigma_*$ is the induced map on homology with  $\RR$ coefficients. Since $(\sigma_*)^2=I$ and $\tr \left(\sigma_*|_{H_2 (S^2 \times S^2)}\right)=-2$, it follows that $\sigma_*=-I$ on $H_2(S^2 \times S^2, \RR )$. Hence each sphere $S^2 \times \{ pt\}$ must be 
 sent  isometrically by $\sigma$ 
to a sphere of the same kind, in an orientation-reversing manner; and the same conclusion similarly applies to spheres of the form 
$\{ pt\} \times S^2$. Since the projection of $S^2 \times S^2$ to either factor is a Riemannian submersion,  it therefore follows
that $\sigma$ must be the product of two isometric, orientation-reversing involutions of $(S^2, g_0)$. However, 
any such involution is diagonalizable, with eigenvalues $\pm 1$. Up to conjugation, the only candidates for these  maps of $S^2$ 
 are therefore 
the involutions  $\mathfrak{a}$ and  $\mathfrak{r}$ described by \eqref{reflect}. However,
$\mathfrak{r}\times \mathfrak{r}$ can be excluded as a candidate for $\sigma$, since it has fixed points. 
 Thus, after interchanging factors if necessary, the only remaining possibilities for $\sigma$ are the 
free anti-holomorphic involutions $\mathfrak{a}\times \mathfrak{r}$ and $\mathfrak{a}\times \mathfrak{a}$ of
$S^2 \times S^2 = \CP_1 \times \CP_1$.

It  therefore only remains to show that $\mathscr{P}:= (S^2\times S^2)/\langle \mathfrak{a}\times \mathfrak{r}\rangle$
is   not diffeomorphic to 
$\mathscr{Q} := (S^2\times S^2)/\langle \mathfrak{a}\times \mathfrak{a}\rangle$. To see this, first notice that 
 $w_2 (\mathscr{Q})\neq 0$, since  the diagonal $S^2 \subset S^2 \times S^2$ projects to   an $\RP^2 \subset \mathscr{Q}$ 
that has   normal bundle $\cong T\RP^2$, and so has self-intersection $\chi (\RP^2) =1$  mod  $2$. By contrast, 
$H_2 (\mathscr{P}, \ZZ_2)$ is generated by the $\RP^2$-image of $S^2 \times \{ (1,0,0)\}$ and the $S^2$-image of $\{ pt\} \times S^2$;
and since each  of these submanifolds has  small perturbations that do not intersect it, both have self-intersection zero, and 
 it follows that  $w_2 (\mathscr{P})=0$. Thus,  the  $4$-manifolds  $\mathscr{P}$ and $\mathscr{Q}$
 certainly aren't  diffeomorphic, 
 because  one is spin, while  the other isn't. 
 \end{proof}

%
%

\subsection{The $b_2(M)=1$ Case}

When $b_2(M)=1$, the  del Pezzo surface
  $(\widetilde{M},J)$ has  degree $c_1^2= 6$. Because  this complex surface has  $K^{-1}$ ample, surface classification  easily allows
  one to show \cite{delpezzo,dolgachev} that it must exactly be the blow-up of $\CP_2$ at three non-collinear points, 
  which we may take to be $[1,0,0]$, $[0,1,0]$,   and $[0,0,1]$. By adjusting our coordinates if necessary, the free anti-holomorphic 
  involution $\sigma : \widetilde{M}\to \widetilde{M}$ can moreover then be identified \cite[p. 60]{ctcdp} with the  map 
  $$
  \Upsilon : \CP_2\# 3 \overline{\CP}_2 \to  \CP_2\# 3 \overline{\CP}_2
  $$
  given   by the 
  conjugated Cremona transformation
 $$
      [ z_1 : z_2 : z_3 ] \longmapsto  [ \frac{1}{\bar{z}_1} : -\frac{1}{\bar{z}_2} : \frac{1}{\bar{z}_3}].
$$

This last
 uniqueness assertion might come as something of a surprise.  For instance, if we blow up $\CP_1 \times \CP_1$ at a generic pair of distinct points  that are  interchanged by $\mathfrak{a}\times \mathfrak{r}$,
the anti-holomorphic involution thereby induced on the blow-up is actually isomorphic to the one we would have produced had we  instead started with 
$\mathfrak{a}\times \mathfrak{a}$; for although identifying the two-point blow-up of $\CP_1 \times \CP_1$  with
the three-point blow up $\CP_2$ in the standard way produces  two  anti-holomorphic involutions that  look superficially different, these
actually turn out to simply  differ by a  Cremona transformation \cite{russo}. 
In particular, it follows  that the non-spin $4$-manifolds $\mathscr{P}\# \overline{\CP}_2$ and $\mathscr{Q}\# \overline{\CP}_2$ are both 
diffeomorphic to $(\CP_2\# 3 \overline{\CP}_2) /\langle \Upsilon \rangle$. 

Our discussion thus far has revealed that 
 any compact  oriented  Einstein manifold $(M^4, h)$ 
with $\pi_1\neq 0$, $b_2=1$,  and $\det (W^+)>0$  must  be   diffeomorphic to $\mathscr{Q}\# \overline{\CP}_2$. 
We will  now show that, conversely,   this possibility actually arises, and   that it  does so moreover in an essentially  unique way:

\begin{prop} There is an Einstein metric $h$ on $\mathscr{Q}\# \overline{\CP}_2$  
that satisfies $\det (W^+)> 0$ at every point. Moreover, any 
 compact oriented  Einstein manifold $(M^4,h^\prime)$ with  $\pi_1\neq 0$,  $b_2=1$, and $\det (W^+)> 0$ is 
isometric  to $(\mathscr{Q}\# \overline{\CP}_2, \mathsf{a} h)$ for some positive constant $\mathsf{a}$. 
As a consequence,   the restricted Einstein 
moduli space $\mathscr{E}_{\det}(\mathscr{Q}\# \overline{\CP}_2)$  therefore consists of exactly one  point.
\label{persius}
\end{prop} 
\begin{proof}
Siu \cite[p. 621]{s} proved that $\CP_2\# 3 \overline{\CP}_2$ admits a $J$-compatible 
K\"ahler-Einstein metric $g$ with Einstein constant $1$ that is 
  invariant under the compact group of automorphisms generated  by the permutations
$$
 \alpha_1 = \left[\begin{array}{ccc} 1&  &  \\ &  &  1\\  &  1& \end{array}\right], 
\quad 
\alpha_2 = \left[\begin{array}{ccc} &  & 1 \\ &  1&  \\ 1 &  & \end{array}\right],
\quad 
\alpha_3 = \left[\begin{array}{ccc} & 1 &  \\1 &  &  \\ &  & 1\end{array}\right] , $$
along with the action of the $2$-torus 
$$
\mathbb{T}^2:=\mathbf{S}( \mathbf{U}(1) \times  \mathbf{U}(1) \times  \mathbf{U}(1)) =\left\{ 
\left[\begin{array}{ccc}e^{i\theta} &  &  \\ & e^{i\phi} &  \\ &  & e^{-i(\theta + \phi)}\end{array}\right]
\right\}~,
$$
lifted in the obvious way to act on  the  three-point blow-up. 
In point of fact, Matsushima's theorem \cite{mats} tells us that invariance under the torus action is   automatic here,
 because $\mathbb{T}^2/\ZZ_3$
is actually the unique maximal compact subgroup of the identity component  $(\CC^\times \times \CC^\times)/\ZZ_3$ of the 
complex automorphism group 
of $\CP_2\# 3 \overline{\CP}_2$.  By contrast, its invariance with respect to the specific  finite group  
 $\mathfrak{S}_3$ generated by the $\{ \alpha_j\}$, together with the normalization of choosing the Einstein constant to be $1$,
 {\em uniquely} picks out Siu's K\"ahler-Einstein metric $g$. Indeed, the Bando-Mabuchi uniqueness theorem \cite{bandomab} tells us that
 any other $J$-compatible K\"ahler-Einstein metric $\widehat{g}$ on $\CP_2\# 3 \overline{\CP}_2$ with Einstein constant $1$ 
 must be obtained from $g$ by 
 moving it  by an element of the connected component of $(\CC^\times \times \CC^\times)/\ZZ_3$ of the 
complex automorphism group. 
However, any such rival Einstein metric $\widehat{g}\neq g$ is then invariant under a different  representation of  $\mathfrak{S}_3$, where the 
generators $\widehat{\alpha}_j= A^{-1}\circ  \alpha_j\circ A$ have been conjugated by a diagonal matrix $A$ of determinant $1$ whose 
eigenvalues do not all have norm $1$. If $\widehat{g}$ were also invariant under the original $\alpha_j$, it would then 
be invariant  $\alpha_j \circ A^{-1}\circ  \alpha_j\circ A\in \CC^\times \times \CC^\times$ for each $j=1,2,3$, and the powers of at least one
such diagonal matrix will then diverge in $\CC^\times \times \CC^\times$. But this is a contradiction,
 since the isometry group of any compact Riemannian
manifold is compact. This proves that Siu's K\"ahler-Einstein metric is uniquely determined by  its $\mathfrak{S}_3$-invariance, 
together with our  (arbitrarily chosen) normalization of its Einstein constant.

 Now notice that 
 $$ \Upsilon \circ \alpha_j  =    \beta_j \circ \alpha_j \circ \Upsilon, \quad j=1,2,3,$$
 where the $\beta_j \in \mathbb{T}^2$ are defined by
 $$
  \beta_1 = \left[\begin{array}{ccc} 1&  &  \\ & -1 &  \\  &  & -1 \end{array}\right], 
\quad 
\beta_2 = \left[\begin{array}{ccc} 1&  &  \\ &  1&  \\  &  & 1\end{array}\right],
\quad 
\beta_3 = \left[\begin{array}{ccc} -1&  &  \\ & -1 &  \\ &  & 1\end{array}\right] . 
 $$
 Since the K\"ahler-Einstein metric $g$ is compatible with both $J$ and $-J$, and since $\Upsilon$
 just interchanges these two integrable complex structures, it follows that $\widehat{g}:=\Upsilon^*g$ is  a $J$-compatible 
 $\lambda=+1$ K\"ahler-Einstein metric. But, using the invariance of $g$ under  $\alpha_j$ and $\beta_j$,  we now see  that 
 \begin{eqnarray*}
\alpha_j^*\widehat{g}&=& \alpha_j^* \Upsilon^* g = (\Upsilon \circ \alpha_j)^*g\\
&=& (\beta_j \circ \alpha_j \circ \Upsilon)^*g = \Upsilon^* (\beta_j \circ \alpha_j )^*g \\
&=& \Upsilon^*g = \widehat{g}.
\end{eqnarray*}
This shows that $\widehat{g}$ is another $\lambda=+1$ K\"ahler-Einstein metric that is invariant under the 
action of $\mathfrak{S}_3$ generated by the $\{ \alpha_j \}$. But since Siu's metric is uniquely characterized by these properties, we must
 have $g = \widehat{g} = \Upsilon^*g$. Thus, the free anti-holomorphic involution  $\Upsilon$ is an isometry 
 of $(\CP_2\# 3 \overline{\CP}_2, g)$, and $g$ therefore descends to $\mathscr{Q}\# \overline{\CP}_2=
 (\CP_2\# 3 \overline{\CP}_2)/\langle \Upsilon \rangle$ as an Einstein metric $h$ with $\det (W^+) > 0$
 everywhere. Moreover, since there is only one del Pezzo surface of degree $6$, 
 Proposition \ref{rule} and  the Bando-Mabuchi uniquess theorem together 
 guarantee that any other compact  oriented  Einstein manifold $(M^4, h^\prime)$ 
with $\pi_1\neq 0$, $b_2=1$,  and $\det (W^+)>0$  must  be isometric to a rescaled version of this Einstein manifold $(\mathscr{Q}\# \overline{\CP}_2, h)$. 
\end{proof}

\subsection{The $b_2(M)=2$ Case}


When $b_2(M)=2$, the  del Pezzo surface
  $(\widetilde{M},J)$ has  degree $c_1^2= 4$.
Because this complex surface has 
 $K^{-1}$ ample, the  Riemann-Roch-Hirzebruch and Kodaira vanishing theorems  
immediately  tell us that $h^0 (\widetilde{M}, \mathcal{O} (K^{-1}))= 5$, and  $h^0 (\widetilde{M}, \mathcal{O} (K^{-2}))= 13$.
On the other hand, surface classification tells us that $\widetilde{M}$ must be obtained by blowing up $\CP_2$ at five distinct points, 
no three of which are collinear.  Using these  facts, one can then deduce \cite{delpezzo,dolgachev} that the anti-canonical system 
embeds $\widetilde{M}$ in $\mathbb{P}([H^0 (\mathcal{O} (K^{-1})]^*)\cong \CP_4$, and that the image of 
$(\widetilde{M},J)$  is actually  the transverse intersection of two non-singular quadrics in $\CP_4$. In our case, though, we
also have an anti-hololomorphic involution $\sigma : \widetilde{M}\to \widetilde{M}$, and this then induces
 a complex-anti-linear involution
$$\sigma^* : [H^0 (\mathcal{O} (K^{-1}))]^*\to [H^0 (\mathcal{O} (K^{-1}))]^*$$
that looks like   component-wise complex conjugation in  $\CC^5$.  Obviously, the  image
of $\widetilde{M}$ is automatically invariant under the involution of $\CP_4$ induced by $\sigma^*$, and this involution 
 moreover restricts to $\widetilde{M}$ 
as  the given  anti-holomorphic involution $\sigma$. In addition, there is an induced complex-anti-linear  involution 
$$\sigma_* : H^0 (\mathcal{O} (K^{-2}))\to H^0 (\mathcal{O} (K^{-2}))$$
that is compatible  with the one induced by $\sigma^*$ on the $15$-dimensional space $\odot^2H^0 (\mathcal{O} (K^{-1}))$
 of homogeneous quadratic polynomials.
The $2$-dimensional kernel of the restriction map $\odot^2H^0 (\mathcal{O} (K^{-1}))\to H^0 (\mathcal{O} (K^{-2}))$
therefore also carries an induced complex conjugation map. Taking a generic real basis for this space, 
we thus see  that $\widetilde{M}\subset \CP_4$ is actually  the transverse intersection of two non-singular quadrics {\em with real coefficients},
but with disjoint real loci. 
By choosing a suitable basis for the real homogeneous polynomials vanishing on  $\widetilde{M}$, and then 
 altering our homogeneous  coordinates by the action of $\mathbf{GL}(5, \RR)$, we may thus arrange for 
 $\widetilde{M}$ to be cut out \cite{ctcpencils,ctcdp} by the equations 
 $$0= \sum_{j=1}^5  z_j^2 = \sum_{j=1}^5  a_j z_j^2$$
where $a_1, \ldots , a_5$ are distinct real numbers. Conversely, any such choice of the coefficients $a_j$
defines a degree-four del Pezzo surface $\widetilde{M}$ with free anti-holomorphic involution $\sigma$;  the requirement
that the coefficients $a_j$ be distinct is  exactly equivalent to requiring that  intersection of the given quadrics be smooth. 
Replacing these  quadrics with linear combinations 
and then rescaling our coordinates 
has the effect of replacing  $a_1, \ldots , a_5$ with their images under a fractional linear transformation of 
$\RR$,  so we may further refine our normal form so   that $a_1=1$, $a_2=2$, $a_3=3$, and $3< a_4 < a_5$. 
This not only shows that the moduli space of smooth degree-four del Pezzo surfaces  with free anti-holomorphic involution is {\em connected} \cite{russo,ctcdp},  but also reveals  that this moduli space 
has real dimension $2$. 
 
Now, every smooth degree-four del Pezzo surface admits a $J$-compatible K\"ahler-Einstein metric \cite{sunspot,ty}. Moreover, 
since there are no non-trivial holomorphic vector fields on such a del Pezzo, the uniqueness theorem of Bando-Mabuchi
guarantees that this $J$-compatible K\"ahler-Einstein metric $g$ is completely {\em unique} once we 
exclude non-trivial constant rescalings by, for example,  normalizing the Einstein constant. 
However, if $g$ is a K\"ahler-Einstein metric, then $\sigma^* g$ is also K\"ahler-Einstein. Moreover, since 
$g$ is compatible with the two integrable almost-complex structures $\pm J$, the same is true of  
$\sigma^* g$, since the anti-holomorphic involution $\sigma$ exactly 
interchanges $J$ and $-J$. Since the Einstein metrics $g$ and $\sigma^* g$ also have the same Einstein constant,
it thus follows that $g = \sigma^* g$. Since  that the Einstein metric $g$ is  therefore $\sigma$-invariant. 
it pushes down to a unique Einstein metric $h$ on $M=\widetilde{M}/\langle\sigma \rangle$. We have thus arranged for 
 $g$ to become the pull-back $\widetilde{h}$ of an Einstein metric $h$ on $M$ with $\det (W^+)> 0$. To summarize: 

\begin{prop} Any compact oriented,  Einstein manifold $(M^4,h)$ with \linebreak  $\pi_1\neq 0$,  $b_2=2$, and  $\det (W^+)> 0$ is 
orientedly diffeomorphic to $\mathscr{Q}\# 2\overline{\CP}_2$, and  is doubly covered by  a degree-four del Pezzo surface equipped 
with a fixed-point-free free anti-holomorphic involution. Moreover, the moduli space $\mathscr{E}_{\det}(\mathscr{Q}\# 2\overline{\CP}_2)$ of these special 
 Einstein metrics  
is non-empty, connected, and of real dimension $2$. 
\label{pencil}
\end{prop}

\subsection{The $b_2(M)=3$ Case}

We finally  come to the  case where $b_2(M)=3$, and where  $(\widetilde{M},J)$ is 
 a   del Pezzo surface
of  degree $c_1^2= 2$. This time,   Riemann-Roch-Hirzebruch and Kodaira vanishing 
 tell us that $h^0 (\widetilde{M}, \mathcal{O} (K^{-1}))= 3$, while the classification of rational surfaces  tells us that  $(\widetilde{M},J)$
 is obtained from $\CP_2$ by blowing up $7$ points, with no three of them  collinear, and no six  on a conic. 
 This can then be used \cite{delpezzo,dolgachev} to show that the anti-canonical system is base-point free, and so 
 defines a degree-$2$ holomorphic map 
 $$\widetilde{M}\to \mathbb{P}([H^0 (\mathcal{O} (K^{-1})]^*)\cong \CP_2;$$ 
further use of the ampleness of $K^{-1}$ then reveals that   $(\widetilde{M},J)$ 
is therefore a  branched double of the projective plane, with  branch locus a smooth quartic curve. 
Thus, $\widetilde{M}$ is biholomorphic to the subvariety of $\mathcal{O}(2) \to \CP_2$,
given by $\zeta^2 = - f(z_1, z_2, z_3)$, where $[z_1,z_2,z_3]\in \CP_2$, the fiber-coordinate $\zeta$ is
homogeneous of degree $2$ in $(z_1,z_2, z_3)$, and where $f\in H^0 (\CP_2 , \mathcal{O}(4))$ vanishes along  a smooth
quartic plane curve $\Sigma$. 

However, in our case, we also have a fixed-point-free anti-holomorphic involution $\sigma: \widetilde{M}\to \widetilde{M}$,
and the induced anti-holomorphic  action of this involution on the line bundle $K^{-1}\to \widetilde{M}$ then 
 induces a standard complex conjugation map on $[H^0 (\mathcal{O} (K^{-1})]^*\cong \CC^3$. The
 induced anti-holomorphic action on $\CP_2$ then preserves the branch locus, and acts on 
 $\Sigma$ without  fixed points. We may thus take the defining equation $f$ of $\Sigma$ to be real,
 and everywhere positive on $\RP^2 \subset \CP_2$. Fortunately, the moduli space
 of such smooth real quartics without real points has been studied extensively, 
 and is know to be connected. Indeed, it can be naturally identified \cite{quartics} with a
 specific arithmetic quotient of hyperbolic $6$-space. 
 
 For each such real quartic curve, we conversely obtain a unique degree-two del Pezzo surface $(\widetilde{M},J)$
 given by  $\zeta^2 = - f(z_1, z_2, z_3)$, and which is equipped a  with
 fixed-point-free anti-holomorphic involution $\sigma : \widetilde{M}\to \widetilde{M}$ given by 
 $(z_1,z_2,z_3, \zeta)\mapsto (\bar{z}_1,\bar{z}_2,\bar{z}_3, \bar{\zeta})$. But any 
 degree-two del Pezzo  admits  \cite{sunspot} a $J$-compatible K\"ahler-Einstein metric $g$, and pulling back
 this metric by our anti-holomorphic involution then gives a second K\"ahler-Einstein metric $\sigma^* g$ 
 with the same Einstein constant. But the del Pezzo surface  $\widetilde{M}$ is also biholomorphic to 
 a blow-up of $\CP_2$ at seven points in general position, it carries no holomorphic vector fields. 
 Thus, the Bando-Mabuchi uniqueness theorem tells us that $\sigma^*g =g$. It therefore follows that 
 $g$ descends to the quotient $M= \widetilde{M}/\langle \sigma \rangle$ as a uniquely defined Einstein metric
 $h$ with $\det (W^+) > 0$, thereby making $g$ equal its pull-back $\widetilde{h}$. We have thus proved the following:

\begin{prop} 
Any compact oriented,  Einstein manifold $(M^4,h)$ with \linebreak  $\pi_1\neq 0$,  $b_2=3$, and  $\det (W^+)> 0$ is 
orientedly diffeomorphic to $\mathscr{Q}\# 3\overline{\CP}_2$, and  is doubly covered by  a degree-two  del Pezzo surface equipped 
with a fixed-point-free free anti-holomorphic involution. Moreover, the moduli space $\mathscr{E}_{\det}(\mathscr{Q}\# 2\overline{\CP}_2)$ of these special 
 Einstein metrics  
is non-empty, connected, and of real dimension $6$. 
\label{branch}
\end{prop}

\subsection{Proofs of the Main Theorems}
\label{mainline} 

By putting together  the above results, it is  now straightforward to prove  our main theorems. For the sake of clarity, we  will 
do so in reverse order.

\medskip

\noindent 
{\sf Proof of Theorem \ref{beta}.} If $(M,h)$ is a compact oriented Einstein $4$-manifold with $\pi_1\neq 0$ and 
$\det (W^+)> 0$,  case \eqref{ahi} of Proposition \ref{rule} tells us that $M= \widetilde{M}/\langle \sigma \rangle$,
where $\widetilde{M}$ is a del Pezzo surface, and $\sigma$ is a fixed-point-free anti-holomorphic involution;
moreover, Theorem \ref{katydid} tells us that the pull-back $\widetilde{h}$ to $\widetilde{M}$ is actually K\"ahler-Einstein.

Because $4 -b_2(M) = \frac{1}{2} c_1^2 (\widetilde{M}) > 0$, the only possible values of 
$b_2(M)$ are $0,1,2$, or $3$,  and we have thoroughly  analyzed each of these possibilities. 
When $b_2(M)=0$, Proposition \ref{twofer} tells us that $M$ must be diffeomorphic to $\mathscr{P}$ or $\mathscr{Q}$;
both cases actually arise, and they are topologically distinct, because only one of them is spin. 
When $b_2(M)=1$, Proposition \ref{persius} tells us that $M$ must be diffeomorphic to $\mathscr{Q} \# \overline{\CP}_2$,
and that this manifold actually carries an Einstein metric of the required type. 
When  $b_2(M)=2$, $M$ must instead be diffeomorphic to $\mathscr{Q} \# 2\overline{\CP}_2$
by Proposition \ref{pencil}, which also tells us that this manifold actually carries a family of Einstein metrics with the required property. 
Finally, when  $b_2(M)=3$,
Proposition \ref{branch} tells us that $M$ is necessarily diffeomorphic to  $\mathscr{Q} \# 3\overline{\CP}_2$,
and that this manifold actually carries   a family of such  Einstein metrics. 
 \hfill $\Box$

\medskip

\noindent 
{\sf Proof of Theorem \ref{alpha}.}  
By Theorem \ref{beta}, there are exactly five diffeotypes of {\em non-simply connected} compact oriented $4$-manifolds $M$
that carry Einstein manifolds with $\det (W^+) > 0$  everywhere; namely, these are 
$\mathscr{P}$ and $\mathscr{Q}\# k \overline{\CP}_2$ for $k=0,1,2,$ and $3$. Moreover, 
Propositions \ref{twofer}, \ref{persius}, \ref{pencil}, and \ref{branch} tell us that in each case  the moduli
space $\mathscr{E}_{\det}(M)$ of these special Einstein metrics is actually connected. 
In addition, there are \cite{lebdet} exactly ten {\em simply connected} diffeotypes of compact oriented $4$-manifolds that
carry such metrics, corresponding to  the ten deformation types of del Pezzo surfaces;  for each 
such diffeotype, our moduli space $\mathscr{E}_{\det}(M)$ of special Einstein metrics is connected, because it 
is in fact 
exactly 
the (connected) moduli space of del Pezzo complex structures, modulo the $\ZZ_2$-action 
induced by $J\rightsquigarrow -J$. 
Taken together, this means there are  exactly 15 possible diffeotypes, and that in each case
the moduli space  $\mathscr{E}_{\det}(M)$ of these special Einstein metrics is both non-empty 
and connected. 

Finally, the moduli space $\mathscr{E}_{\det}(M)$ of 
these special Einstein  metrics is always both open and closed as a subset of the  moduli space $\mathscr{E}(M)$
of {\em all} Einstein metrics. Indeed, by \cite{lebdet}, these special Einstein metrics  are characterized 
among all Einstein metrics by the open condition 
$\det (W^+) > 0$; whereas  results of Derdzi\'nski \cite[Theorem 2]{derd} and Hitchin \cite[Theorem 13.30]{bes} instead characterize them, 
among Einstein metrics on these spaces,  
 by 
the pair of  closed conditions $\det (W^+) = \frac{1}{3\sqrt{6}} |W^+|^3$ and $s\geq 0$. Thus, for  each of these fifteen $4$-manifolds $M$, 
 the connected
space $\mathscr{E}_{\det}(M)$ is precisely a single connected component of the Einstein moduli space  $\mathscr{E}(M)$. 
 \hfill $\Box$

\section{Related Results} 
\label{last} 

For clarity and simplicity, we have  supposed throughout this article 
that the Einstein metrics $h$ under investigation satisfied  Wu's condition 
$\det (W^+) > 0$. However,  by  \cite[Theorem C]{lebdet},  we could have actually  reached exactly the same conclusions 
if we had  merely imposed an ostensibly  weaker hypothesis:

\begin{thm}
\label{delsegno}
Let $(M,h)$ be a compact 
 oriented Einstein  $4$-manifold. If 
 \begin{equation}
\label{home}
W^+\neq 0 \quad \mbox{and}\quad  \det (W^+) \geq - \frac{5\sqrt{2}}{21\sqrt{21}}|W^+|^3
\end{equation}
 at every point of  $M$, then  $(M,h)$ actually satisfies  $\det (W^+) > 0$ everywhere. Consequently,  by Theorem \ref{alpha}, there 
 are exactly 15 diffeotypes of $4$-manifolds $M$ that  carry such Einstein metrics, and their  moduli space $\mathscr{E}_{\det}(M)$ 
 is in each case exactly a connected component of the  Einstein moduli space $\mathscr{E}(M)$. 
\end{thm}

In  fact, the results of \cite{lebdet}  apply more generally to oriented  Riemannian $4$-manifolds that 
satisfy  $\delta W^+=0$ and \eqref{home}. This led there to a complete diffeomorphism classification of all such manifolds 
with  $b_+\neq 0$. Regarding  the $b_+=0$, 
 case, we can now at least say the following: 
  
%
%

\begin{thm}
\label{alfine}
Let $(M,h)$ be a compact 
 oriented Riemannian  $4$-manifold with $\delta W^+=0$, and suppose that $b_+(M)=0$. If
 $h$ satisfies \eqref{home} at every point, then  $M$ admits a double 
 cover $\widetilde{M}$ that is 
 diffeomorphic to $(\Sigma_{\mathfrak{g}} \times S^2) \# 2k \overline{\CP_2}$, where $k$ and $\mathfrak{g}$ are   non-negative   integers,
 and where $\Sigma_{\mathfrak{g}}$ is the
 compact oriented surface of genus $\mathfrak{g}$. Conversely, each of these possibilities occurs:
 for every pair $(k,\mathfrak{g})$ of non-negative integers, 
 there is a  compact oriented $4$-manifold $(M,h)$ with $b_+=0$, $\delta W^+=0$, and $\det (W^+) > 0$
that  is doubly covered by a manifold $\widetilde{M}$ diffeomorphic to 
 $(\Sigma_{\mathfrak{g}} \times S^2) \# 2k \overline{\CP_2}$. 
\end{thm}
\begin{proof}
By \cite[Proposition 3.2]{lebdet}, any such manifold $(M,h)$ is double-covered by a 
complex surface $(\widetilde{M}, J)$ on which the 
pull-back $\widetilde{h}$  of $h$ becomes conformal to a $J$-compatible  K\"ahler metric $g$ of positive scalar 
curvature $s$. 
By \cite{yauruled}, this implies that $(\widetilde{M}, J)$  has Kodaira dimension $-\infty$, and so is rational or ruled. 
Since  the signature  $\tau (\widetilde{M})= 2 \tau (M)$ must be even, it follows that $\widetilde{M}$ is diffeomorphic to either 
 $(\Sigma_{\mathfrak{g}} \times S^2) \# 2k \overline{\CP_2}$ for some  $\mathfrak{g}, k \geq 0$, where $\Sigma_{\mathfrak{g}}$
 denotes the compact oriented surface of genus $\mathfrak{g}$, or  to the 
  non-spin oriented $2$-sphere bundle 
 $\Sigma_{\mathfrak{g}} \rtimes S^2$ over some $\Sigma_{\mathfrak{g}}$. However, the latter is excluded here 
  by a variant of the proof of Lemma \ref{halte}; indeed, since the putative 
  oriented $4$-manifold $M= [\Sigma_{\mathfrak{g}} \rtimes S^2]/\ZZ_2$ would  have signature $\tau=0$ and 
 $b_+=0$, its second cohomology group $H^2(M, \ZZ)$ would be finite, and, since $\mbox{Tor } H^2(\Sigma_{\mathfrak{g}} \rtimes S^2,\ZZ)
 =   \mbox{Tor }  H_1 (\Sigma_{\mathfrak{g}} \rtimes S^2, \ZZ)= 
 \mbox{Tor }  H_1 (\Sigma_{\mathfrak{g}}, \ZZ) =0$,   
 pulling back a 
 spin$^c$ structure on $M$ would  then  yield  a  spin structure on the non-spin manifold $\widetilde{M}\approx \Sigma_{\mathfrak{g}} \rtimes S^2$. This contradiction therefore shows  that $\widetilde{M}$
 must instead be diffeomorphic to $(\Sigma_{\mathfrak{g}} \times S^2) \# 2k \overline{\CP_2}$ for some 
 $\mathfrak{g}, k \geq 0$. 

Conversely, let   $\breve{\Sigma}= \#_{\mathfrak{g}+1}\mathbb{RP}^2$ 
be the connected sum of $\mathfrak{g}+1$ copies of the real projective
plane, and let $\breve{g}_1$ be a smooth Riemannian metric on $X$ that has positive Gauss curvature on some 
 non-empty open set  $\breve{U}\subset X$.  Let $\Sigma\to \breve{\Sigma}$
 be the oriented double cover of $\breve{\Sigma}$, let $g_1$ be the 
 pull-back of $\breve{g}_1$ to $\Sigma$,
and let \linebreak $j:T\Sigma \to T\Sigma$ be the  integrable almost-complex complex structure on $\Sigma$ induced by 
 $g_1$ and the orientation of $\Sigma$. The non-trivial deck transformation $\rad : \Sigma \to \Sigma$ now becomes
 a fixed-point-free    anti-holomorphic involution of the  genus-$\mathfrak{g}$ compact complex curve $(\Sigma, j)$, 
 while $g_1$ becomes a 
 a  $j$-compatible K\"ahler metric on $\Sigma$ that 
  has  Gauss curvature $\kappa > 0$ on the non-empty  $\rad $-invariant region $U$ that is  
 the pre-image of $\breve{U}$. If $g_0$ is the usual  unit-sphere metric on $S^2 = \CP_1$, 
 then the Riemannian product $(\Sigma, g_1) \times (S^2, \varepsilon g_0)$ will have  positive scalar 
 curvature provided we take $\varepsilon > 0$ to be small enough so that ${\varepsilon}^{-1} > -\min \kappa$.
 Moreover,  the product  K\"ahler metric $g_1\oplus \varepsilon g_0$ on $\Sigma \times \CP_1$
  has positive holomorphic section curvature on the open subset $U\times \CP_1$. 
 Let us  now endow $\Sigma\times \CP_1$ with the fixed-point-free anti-holomorphic involution $\rid := \rad \times \mathfrak{a}$, 
 where $\mathfrak{a}: S^2 \to S^2$ is once again the antipodal map of \eqref{reflect}, and,  for a given  $k\geq 0$, 
 choose $2k$ distinct points  $p_1, \ldots , p_{2k}\in U \times \CP_1$ such that $\rid (p_{2j-1}) = p_{2j}$ for $j=1, \ldots , k$. 
Letting  $\widetilde{M}$ denote  
 the blow-up of $\Sigma\times \CP_1$ at $p_1, \ldots , p_{2k}$, there is then a fixed-point-free anti-holomorphic involution 
 $\sigma : \widetilde{M}\to \widetilde{M}$ induced by $\rid$, and, using a result of Hitchin \cite{hitpos}, we  will now 
 construct a $\sigma$-invariant K\"ahler metric $\widetilde{g}$ of positive scalar curvature on $\widetilde{M}$. The recipe only calls for 
 changing  the previously constructed K\"ahler metric $g_1\oplus \varepsilon g_0$ within a disjoint union of the Riemannian balls
 of radius $\epsilon$ about the $p_j$, and modifies 
  the metric within these balls  by adding $i\mathbf{t}\partial \bar{\partial} f (\mathsf{r})$ to the K\"ahler form,
  where  $\mathsf{r}$ is the Riemannian distance from the center $p_j$, 
$f (\mathsf{r})$ is a gently modified version of  $\log \mathsf{r}$ that is constant for $\mathsf{r}> \epsilon$,
and the parameter $\mathbf{t}$ is any sufficiently  small positive constant. For  $\mathbf{t}>0$ sufficiently small,   Hitchin's computation  then shows that 
this modified  K\"ahler metric $\widetilde{g}$ has  positive  scalar curvature everywhere,  because 
our background metric has positive  holomorphic sectional curvature in the modification region, and positive  scalar curvature 
 everywhere else. On the other hand, since we have carefully arranged for this K\"ahler metric 
$\widetilde{g}$ on the blow-up to be  $\sigma$-invariant, our $\widetilde{g}$ is   automatically   the pull-back of a unique Riemannian metric 
$g$  on the oriented $4$-manifold $M:= [(\Sigma\times \CP_1) \# 2k \overline{\CP}_2]/\langle \sigma \rangle$. However, 
by construction, $g$ has scalar curvature $s> 0$, and is moreover everywhere locally isometric to a K\"ahler metric $\widetilde{g}$. 
An observation of Derdzi\'nski \cite{bes,derd}
thus shows that  the conformally rescaled metric $h= s^{-2}g$  therefore has   $\det (W^+) > 0$ and  $\delta W^+=0$ at every point of $M$. 
\end{proof}

While the above enumeration of the  possibilities for   $\widetilde{M}$  is similar in  spirit
to our discussion of the  Einstein case, Theorem \ref{alfine}  is certainly far  weaker than our main results. First of all, we have not tried to classify the possible $\ZZ_2$-actions that arise, although it seems clear that that there  must be many of them. 
Second,  in stark contrast to the Einstein case, 
the moduli spaces of solutions to the  weaker  equation $\delta W^+=0$ consistently turn out to be 
{\em infinite dimensional} in the present context, and nothing substantial seems to be known  concerning whether or not they are connected. 
We leave these open questions for the  reader's further consideration, in the hope that this will  stimulate further research, 
and eventually  lead to definitive answers.

\pagebreak 
%

\end{document}